\documentclass[reqno]{amsart}

\usepackage{amsmath,mathrsfs,amssymb,multirow,setspace,enumerate}
\usepackage{graphicx}
\usepackage{tikz}
\usepackage[colorlinks=true,allcolors=black]{hyperref}
\usepackage{relsize}

\theoremstyle{plain}
\usepackage[a4paper, total={7in, 9in}]{geometry}
\newtheorem{theorem}{Theorem}[section]
\newtheorem{lemma}[theorem]{Lemma}
\newtheorem{proposition}[theorem]{Proposition}
\newtheorem{corollary}[theorem]{Corollary}

\theoremstyle{definition}

\theoremstyle{remark}
\newtheorem{remark}[theorem]{Remark}

\input xy
\xyoption{all}

\begin{document}
	\title[The cyclic graph of a semigroup]{The cyclic graph of a semigroup}
	\author[Sandeep Dalal, Jitender Kumar, Siddharth Singh]{Sandeep Dalal, Jitender Kumar, Siddharth Singh}
	\address{Department of Mathematics, Birla Institute of Technology and Science Pilani, Pilani, India}
	\email{deepdalal10@gmail.com,jitenderarora09@gmail.com,sidharth$\_$0903@hotmail.com}
	
\begin{abstract}
 The cyclic graph $\Gamma(S)$ of a semigroup $S$ is the simple graph  whose vertex set is $S$ and two vertices $x, y$ are adjacent if the subsemigroup generated by $x$ and $y$ is monogenic. In this paper, we classify the  semigroup $S$ such that whose cyclic graph $\Gamma(S)$ is complete, bipartite, tree, regular  and a null graph, respectively. Further, we determine the clique number of $\Gamma(S)$ for an arbitrary semigroup $S$. We obtain the independence number of $\Gamma(S)$ if $S$ is a finite monogenic semigroup. At the final part of this paper, we give   bounds for independence number of $\Gamma(S)$ if $S$ is a semigroup of bounded exponent and we also characterize the semigroups attaining the bounds. 
\end{abstract}

	
	
	\maketitle
 
\section{Introduction}
The investigation of graphs associated to semigroups is a large research area. In 1964, Bosak \cite{b.bosak1964graphs} studied certain graphs over semigroups. There are various graphs associated with semigroups viz. Cayley graphs (\cite{a.brandt,a.KhosraviCayleygraphs2012}), directed power graphs  \cite{a.kelarev2002directed,a.kelarev2001powermatrices}, undirected power graphs  \cite{a.MKsen2009,a.shitov}, commuting graphs \cite{a.Araujo2015,a.Araujo2011,a.Shitov2020}. A significant number of research papers devoted to cyclic graphs overs groups and semigroups. In 2007, Abdollahi et al. \cite{a.Abdollahi2007-noncyclic} introduced the noncyclic graph\index{noncyclic graph} of a group $G$. The noncyclic graph of a group $G$ is the simple undirected graph whose vertex set is $G \setminus T$, where $T = \{x \in G :  \; {\rm the \; subgroup \; generated \; by} \; x, y \; {\rm is \; cyclic \; for \; all} \; y \in G \}$ and two distinct vertices $x, y$ are adjacent if the subgroup generated by $x, y$ is not cyclic. The complement of noncyclic graph of $G$ is called the cyclic graph of $G$. Further, the notion of cyclic graph was slightly modified by Ma et al. \cite{a.Ma2013} as follows. The cyclic graph $\Gamma(G)$ is the simple undirected graph whose vertex set is $G$ and two distinct vertices $x$ and $y$ are adjacent if the subgroup generated by $x, y$ is cyclic. Ma et al. \cite{a.Ma2013} studied the graph-theoretic properties of cyclic graph of a finite group namely, bipartite, diameter, clique number etc. Particularly, they have investigated the graph invariants of the cyclic graph $\Gamma(G)$ when $G$ is dihedral group $D_{2n}$ or dicyclic group $Q_{4n}$.  Also, \cite{a.Ma2013} proved that ${\rm Aut}(G) = {\rm Aut}(\Gamma(G))$ if and only if $G$ is a Klein group $\mathbb Z_2 \times \mathbb Z_2$. The definition of cyclic graph has been extended to semigroups by Afkhami et al. \cite{a.afkhami2014cyclic}. The cyclic graph $\Gamma(S)$ is the simple undirected graph whose vertex set is the semigroup $S$ and two distinct vertices $x$ and $y$ are adjacent if $\langle x, y \rangle$ is a monogenic  subsemigroup of $S$. Afkhami et al. \cite{a.afkhami2014cyclic} provided the structure of $\Gamma(S)$, where $S$ is a finite semigroup and discussed the graph-theoretic properties of $\Gamma(S)$, viz. planarity, genus, girth etc.

This paper is structured as follows.  In Section $2$, we provide necessary background material and fix our
notations used throughout the paper. In Section $3$, we provide the structure of $\Gamma(S)$ and then some basic properties viz. bipartite, completeness, tree, regular, null graph etc. are discussed. Section $4$ investigates the clique number of $\Gamma(S)$ and classifies the semigroup $S$ such that the clique number of  $\Gamma(S)$ is finite. In  Section $5$, we examine the independence number of $\Gamma(S)$.

\section{Preliminaries}
In this section, we recall necessary definitions, results and notations of graph theory from \cite{b.West} and semigroup theory from \cite{b.Howie}. 
A graph $\mathcal{G}$ is a pair  $ \mathcal{G} = (V, E)$, where $V = V(\mathcal{G})$ and $E = E(\mathcal{G})$ are the set of vertices and edges of $\mathcal{G}$, respectively. We say that two different vertices $a, b$ are $\mathit{adjacent}$, denoted by $a \sim b$, if there is an edge between $a$ and $b$. We are considering simple graphs, i.e. undirected graphs with no loops  or repeated edges. If $a$ and $b$ are not adjacent, then we write $a \nsim b$. The \emph{neighbourhood} $ N(x) $ of a vertex $x$ is the set all vertices adjacent to $x$ in $ \mathcal G $. Additionally, we denote $N[x] = N(x) \cup \{x\}$. A subgraph  of a graph $\mathcal{G}$ is a graph $\mathcal{G}'$ such that $V(\mathcal{G}') \subseteq V(\mathcal{G})$ and $E(\mathcal{G}') \subseteq E(\mathcal{G})$. A \emph{walk} $\lambda$ in $\mathcal{G}$ from the vertex $u$ to the vertex $w$ is a sequence of  vertices $u = v_1, v_2,\cdots, v_{m} = w$ $(m > 1)$ such that $v_i \sim v_{i + 1}$ for every $i \in \{1, 2, \ldots, m-1\}$. If no edge is repeated in $\lambda$, then it is called a \emph{trail} in $\mathcal{G}$. A trail whose initial and end vertices are identical is called a \emph{closed trail}. A walk is said to be a \emph{path} if no vertex is repeated.  The length of a path is the number of edges it contains. If $U \subseteq V(\mathcal{G})$, then the  subgraph of $\mathcal{G}$ induced by  $U$ is the graph $\mathcal{G}'$ with vertex set $U$, and with two vertices adjacent in $\mathcal{G}'$ if and only if they are adjacent in $\mathcal{G}$. A graph  $\mathcal{G}$ is said to be \emph{connected} if there is a path between every pair of vertex. A graph $\mathcal{G}$ is said to be \emph{complete} if any two distinct vertices are adjacent.  A path that begins and ends on the same vertex is called a \emph{cycle}.

A \emph{clique}\index{clique} of a graph $\mathcal{G}$ is a complete subgraph of $\mathcal{G}$ and the number of vertices in a  clique of maximum size is called the \emph{clique number} of $\mathcal{G}$ and it is denoted by $\omega({\mathcal{G}})$. An \emph{independent set}\index{independent set} of a graph $\mathcal{G}$ is a subset  of $V(\mathcal{G})$ such that no two vertices in the subset are adjacent in $\mathcal{G}$. The \emph{independence number}\index{independence number} of  $\mathcal{G}$ is the maximum size of an independent  set, it is denoted by $\alpha(\mathcal{G})$. A graph $\mathcal{G}$ is said to be \emph{bipartite}\index{bipartite} if $V(\mathcal{G})$ is the union of two disjoint independent sets.  A graph $\mathcal{G}$ is called a \emph{complete bipartite}\index{complete! bipartite} if $\mathcal{G}$ is bipartite with $V(\mathcal{G}) = A \cup B$, where $A$ and $B$ are disjoint independent sets such that $x \sim y$ if and only if $x \in A$ and $y \in B$. We shall denote it by $K_{n, m}$, where $|A| = n$ and $|B| = m$. A graph $\mathcal{G}$ is said to be a \emph{star graph}\index{star graph} if $\mathcal{G} = K_{1, n}$ for some $n \in \mathbb N$.

\begin{theorem}[{\cite[Theorem 1.2.18]{b.West}}]\label{ch1-bipartite}
A graph  $\mathcal{G}$ is bipartite if and only if $\mathcal{G}$ does not contain an odd cycle.
\end{theorem}


\begin{theorem}[{{\cite[Theorem 1.2.26]{b.West}}}]\label{Eulerian}
A connected graph is Eulerian if and only if its every vertex is of even degree.
\end{theorem}

Now we anamnesis some basic definitions and results on semigroups. A \emph{semigroup} is a non-empty set $S$ together with an associative binary operation on $S$. We say $S$ to be a \emph{monoid}\index{Monoid} if it contains an identity element $e$. A monoid $S$ is said to be a \emph{group} if for each $x$ there exists $x^{-1} \in S$ such that $xx^{-1} = x^{-1}x = e$. A \emph{subsemigroup} of a semigroup is a subset that is also a semigroup under the same operation. A subsemigroup of $S$ which is a group with respect to the multiplication inherited from $S$ will be called  \emph{subgroup}. A non-empty subset $I$ of a semigroup $S$ is said to be an \emph{ideal} of $S$ if $SIS \subseteq I$. An element $a$ of a semigroup $S$ is \emph{idempotent}\index{idempotent element}   if $a^2 = a$ and the set of all idempotents in $S$ is denoted by $E(S)$. A \emph{band}\index{band} is a semigroup  in which every element is idempotent. For a subset $X$ of a semigroup $S$, the subsemigroup generated by $X$, denoted by $\langle X \rangle$, is the intersection of all the subsemigroup of $S$ containing $X$ and it is the smallest subsemigroup of $S$ containing $X$. The subsemigroup $\langle X \rangle$ is the set of all the elements in $S$ that can be written as finite product of elements of $X$. If $X$ is finite then $\langle X \rangle$ is called finitely generated subsemigroup of $S$. A semigroup $S$ is called \emph{monogenic}\index{monogenic semigroup} if there exists $a \in S$ such that $S = \langle a \rangle$. Clearly, $\langle a \rangle = \{a^m \; : \; m \in \mathbb{N}\}$, where $\mathbb{N}$ is the set of positive integers. A subgroup generated by $X$ can be defined analogously. If $X = \{a\}$, then the subgroup generated by $X$ is called \emph{cyclic}\index{cyclic subgroup}. Note that the cyclic subgroup generated by $a$ is $\langle a \rangle = \{a^m : \; m \in \mathbb Z \}$.

For  $X \subseteq S$, the number of elements in $X$ is called the order (or size) of $X$ and it is denoted by $|X|$. The $\mathit{order}$ of an element $a\in S$, denoted by $o(a)$, is defined as  $|\langle a \rangle|$. The set $\pi(S)$ consists order of all the elements of a semigroup $S$. In case of finite monogenic semigroup, there are repetitions among the powers of $a$. Then the set
\[\{x \in \mathbb{N} : (\exists \; y \in \mathbb{N}) a^x = a^y, x \ne y\}\]
is non-empty and so has a least element. Let us denote this least element by $m$ and call it the \emph{index}\index{index} of the element $a$. Then the set
\[\{x \in \mathbb{N} \; : \; a^{m + x} = a^m \}\]
is non-empty and so it too has a least element $r$, which we call the \emph{period}\index{period} of $a$. Let $a$ be an element  with index $m$ and period $r$. Thus, $a^m = a^{m + r}$. It follows that $a^m = a^{m + qr}$ for all $q \in \mathbb{N}$. By the minimality of $m$ and $r$ we may deduce that the powers
$a, a^2, \ldots, a^m, a^{m + 1}, \ldots, a^{m + r-1}$
are all distinct. For every $s \ge m$, by division algorithm we can write $s = m + qr + u$, where $q \ge 0$ and $0 \le u \le r-1$. then it follows that

\[a^s = a^{m + qr}a^u = a^m a^u = a^{m+u}.\]
Thus, $\langle a \rangle = \{a, a^2, \ldots, a^{m + r-1}\}$ and $o(a) = m + r - 1$. The subset \[\mathcal{K}_a = \{a^m, a^{m+1}, \ldots, a^{m+r-1} \}\]  is a subsemigroup of $\langle a \rangle$.  Moreover, there exists $g \in  \mathbb N$ such that $0 \leq g \leq r-1$ and $m + g \equiv 0({\rm mod} \; r)$. Note that $a^{m + g}$ is the idempotent element and so it is the identity element of $\mathcal{K}_a$. Because $a^{(m + g)^2} = a^{2m + 2g} = a^{m + (m +g) + g} = a^{m + tr + g}$ as $m + g \equiv 0({\rm mod} \; r)$ which gives $a^{(m + g)^2} = a^{m +g}$. If we choose $ g' \in  \mathbb N$ such that 
\[0 \leq  g' \leq r-1 \; {\rm and} \; m + g' \equiv 1({\rm mod} \; r),\]

\noindent then $k(m + g') \equiv k ({\rm mod} \; r)$ for all $k \in \mathbb N$, and so the powers $(a^{m + g'})^k$ of $a^{m +g'}$ for $k = 1, 2, \ldots, r$, deplete $\mathcal{K}_a$. Thus, $\mathcal{K}_a$ is the cyclic subgroup of order $r$, generated by $a^{m +g'}$.
Let $a$ be an element of a semigroup $S$ with index $m$ and period $r$. Then the monogenic semigroup $\langle a \rangle$ is denoted by $M(m, r)$. Also, sometimes $M(m, r)$ shall be written as
$\langle a : a^m = a^{m + r}\rangle$. The notations $m_a$ and $r_a$ denotes the index and period of $a$ in $S$, respectively. It is easy to observe that index of every element in a finite group $G$ is one. Consequently, for $a \in G$, we have $\langle a \rangle$ is the cyclic subgroup of $G$. 
 \begin{remark}\label{ch1-Ka^i}
 Let $S = M(m, r) = \langle a \rangle$ be a monogenic semigroup. Then $\mathcal{K}_{a^i} = \langle a^i \rangle \cap \mathcal{K}_a$.
 \end{remark}

A \emph{maximal monogenic subsemigroup} of $S$ is a monogenic subsemigroup of $S$ that is not properly contained in any other monogenic subsemigroup of $S$. We shall denote $\mathcal{M}$  by the set of all elements of $S$ that generates maximal monogenic subsemigroup of $S$ i.e.
$\mathcal{M} = \{a \in S : \; \langle a \rangle \; {\rm is \; a \; maximal \; monogenic \; subsemigroup \; of} \; S \}$. Similarly, a \emph{maximal cyclic subgroup}\index{} of $S$ is a cyclic subgroup of $S$ that is not properly contained in any other cyclic subgroup of $S$. We shall denote $\overline{\mathcal{M}}$  by the set of all elements of $S$ that generates maximal cyclic subgroup of $S$ i.e.
\[\overline{\mathcal{M}} = \{a \in S : \; \langle a \rangle \; {\rm is \; a \; maximal \; cyclic \; subgroup \; of} \; S \}. \]

\emph{Green's relations}\index{Green's relations} were introduced by J.A  Green in $1951$ that characterize the elements of $S$  in terms of principals ideal. They become a standard tool for investigating the structure of semigroup. These relations are defined as
\begin{enumerate}
	\item $ x\; \mathcal L \; y$ if and only if $S^1 x = S^1 y$.
	\item $x \; \mathcal R  \; y$ if and only if $x S^1  = y S^1 $.
	\item $x \; \mathcal J \;  y$ if and only if $S^1 x S^1 = S^1 y S^1$.
	\item $x \; \mathcal H \; y$ if and only if $x \; \mathcal L  \; y$ and $x \; \mathcal R \; y$.
	\item $x \; \mathcal D \; y$ if and only if $x \; \mathcal L \; z$ and $z \; \mathcal R \; y$ for some $z \in S$.
\end{enumerate}
\vspace{.2cm}

\begin{remark}[{{\cite[p. $46$]{b.Howie}}}]\label{ch1-re-Green's-class}
Let $G$ be a group. Then $\mathcal L = \mathcal R = \mathcal H = \mathcal D = \mathcal J = G \times G.$
\end{remark}

\begin{corollary}[{{\cite[Corollary 2.2.6]{b.Howie}}}]\label{ch1-H-class}
	Let $S$ be a semigroup and $f$ be an idempotent element of $S$. Then the $\mathcal{H}$-class $H_f$ containing $f$ is a  subgroup of $S$. 
\end{corollary}

A semigroup $S$ without zero is said to be \emph{simple}\index{simple semigroup} if it has no proper ideals. A semigroup $S$ with zero is called \emph{0-simple}\index{0-simple semigroup} if
(i) $\{0\} \; {\rm and} \; S$   are   its  only  ideals  and  (ii)  $S^2 \neq \{0\}$. A nonzero idempotent in a semigroup $S$ is said to be \emph{primitive} if it is a minimal element in $E(S) \setminus \{0\}$ with respect to the partial order relation $\le$ on $E(S)$ defined by, for $a, b \in E(S), a \le b \; \Longleftrightarrow \;  ab = ba = a$. A semigroup $S$ is said to be \emph{completely $0$-simple} if it is $0$-simple  and has a primitive idempotent. Let $G$ be a group and let $I, \Lambda$ be non-empty sets. Let $P = (p_{\lambda i})$ be a $\Lambda \times I$ matrix with entries in $G^0 ( = G \cup \{0\})$, and suppose $P$ is \emph{regular}\index{regular matrix}, in the sense that no row or column of $P$ consists entirely of zeros. Let $S =  \mathfrak{M}^0[G, I, \Lambda, P] = (I \times G \times \Lambda) \cup \{0\}$, and define a composition on $S$ by 
\begin{equation}
	(i, a, \lambda)(j, b, \mu) =   \left\{ \begin{array}{ll}
		(i, ap_{\lambda j}b, \mu) & {\rm if} \; p_{\lambda j} \ne 0\\
		0 & {\rm if} \; p_{\lambda j} =  0
	\end{array} \right. 
\end{equation}
\[(i, a, \lambda)0 = 0(i, a, \lambda) = 0.\]

\begin{theorem}[{\cite[Theorem 3.2.3]{b.Howie}}]
	Let $S$ be a semigroup. Then $S$ is a completely $0$-simple semigroup if and only if $S \cong \mathfrak{M}^0[G, I, \Lambda, P]$ for some non-empty index sets $I$, $\Lambda$, regular matrix $P$ and a group $G$.
\end{theorem} 


A semigroup is said to be \emph{completely regular}\index{completely regular semigroup} if every element $a$ of $S$ lies in a subgroup of $S$. Further, we have the following characterization of completely regular semigroup.

\begin{proposition}[{{\cite[Proposition 4.1.1]{b.Howie}}}]\label{ch1-completely-regular}
A semigroup $S$ is completely regular if and only if every $\mathcal{H}$-class in $S$ is a group.
\end{proposition}

A semigroup $S$ is said to be of \emph{bounded exponent}\index{bounded exponent} if there exists a positive integer $n$ such that for all $x \in S$, $x^n = f$ for some $f \in E(S)$. If $S$ is of bounded exponent then the \emph{exponent}\index{exponent} of $S$ is the least $n$ such that for each $x \in S$,  $x^n = f$ for some $f \in E(S)$. Note that every finite semigroup is of bounded exponent. We often use the following fundamental properties of semigroups without referring to it  explicitly. Let $S$ be a semigroup of bounded exponent. For $f \in E(S)$, we define
\begin{align}\label{ch1-eq-1}
	S_f = \{ a \in S \; : \; a^m= f \; \text{for some} \; m \in \mathbb N\}.
\end{align}

The following remark is useful in the sequel.

\begin{remark}\label{ch1-re-classification-compoenent} Let $S$ be a semigroup of bounded exponent. Then $S = \underset{f \in E(S)}{\bigcup S_f}$ and for distinct $f, \;  f' \in E(S)$, we have $S_f \bigcap S_{f'} = \emptyset$.
\end{remark}

\section{Graph invariants of $\Gamma(S)$}
In this section,  first we describe the structure of $\Gamma(S)$. Then we investigate graph invariants viz. dominance number, completeness etc. Also, we classify the semigroup $S$ such that $\Gamma(S)$ is bipartite, acyclic, tree and regular, respectively. Finally, we characterize the isolated vertices of $\Gamma(S)$. For $x \in S$ and $m,n \in \mathbb N$, we define 
\[S(x,m,n) = \{y \in S : \; x^m = y^n \}\]
and we write  $C(x) = \bigcup\limits_{m,n \in  \mathbb N}S(x,m,n)$. The
following proposition describes the structure of $\Gamma(S)$.
\begin{proposition}\label{ch2-component-infinite}
	The set $C(x)$ is a connected component of $\Gamma(S)$. Moreover, the components of the graph $\Gamma(S)$ are precisely $\{C(x) \; | \; x \in S \}$.
\end{proposition}
\begin{proof}
	Let $y, z \in C(x)$. Then $y \in S(x,m,n)$ and $z \in S(x,p,q)$.  It follows that $x^m = y^n$ and $x^p = z^q$. Note that $y \sim x^{mp} \sim z$. Thus, $C(x)$ is connected in $\Gamma(S)$. Now suppose that the element $z$ of $S$ is adjacent to a vertex $y$ in $C(x)$. Since $y \sim z$ implies $\langle y, z \rangle = \langle t \rangle$ for some $t \in S$. For $y \in  C(x)$, we have $x^m = y^n$ for some $m, n \in \mathbb N$.  It follows that $y= t^{\alpha}, \; z = t^{\beta}$ and  $z^{n\alpha} = x^{m\beta}$. Thus, $z \in  C(x)$. Hence,  $C(x)$ is a connected component of $\Gamma(S)$.
\end{proof}

In view of Equation (\ref{ch1-eq-1}), we have the following corollary.

\begin{corollary}\label{ch2-component-finite}
	For $f \in E(S)$, we have $C(f) = S_f$. Moreover, if $x \in S$ such that $o(x)$ is finite then $x \in S_f$ for some $f \in E(S)$.
\end{corollary}

Let $S$ be a semigroup of exponent $n$. Each connected component of $\Gamma(S)$ is of the form $S_f$ for some $f \in E(S)$ (see Remark \ref{ch1-re-classification-compoenent} and Corollary \ref{ch2-component-finite}). Since $S_f$ is a connected component and $f \sim x$ for all $x \in S_f \setminus \{f\}$, therefore $E(S)$ is a dominating set of $\Gamma(S)$. Consequently, we have the following lemma.

\begin{lemma}
	Let $S$ be a semigroup of exponent $n$. Then the dominance number of $\Gamma(S)$ is equals to the number of idempotents in $S$.
\end{lemma}

Let $x \in S$ such that $x \ne x^2$. Then $x \sim x^2$ in $\Gamma(S)$. Consequently, we have the following corollary.

\begin{corollary}
	For any semigroup $S$, $\Gamma(S)$ is a null graph if and only if $S$ is a band. 
\end{corollary}

\begin{lemma}\label{ch1-cyclic implies monogenic}
	Every finite cyclic subgroup of a semigroup $S$ is a monogenic subsemigroup of $S$.
\end{lemma}

\begin{proof}
	Let $H$ be a cyclic subgroup of $S$. Then $H = \langle a \rangle$ for some $a \in S$. Since $H$ is finite so that $o(a) = n$ for some $n \in \mathbb N$. Thus $a^n = e$, where $e$ is the identity element of $H$. Consequently, $a^{-1} = a^{n-1}$. Now for any non-negative integer $k$, we get $a^{-k} = a^{k(n-1)}$. Thus, every element of $H$ is a positive power of $a$. Hence, $H$ is a monogenic subsemigroup of $S$.
\end{proof}

\begin{theorem}\label{gamma-complete}
Let $S$ be a semigroup. Then $\Gamma(S)$ is complete\index{complete} if and only if one of the following holds:
\begin{enumerate}[\rm (i)]
		\item $S = \langle a  : a^{1 + r} = a \rangle$
		\item $ S =  \langle a  : a^{2+r} = a^2 \rangle$
		\item $ S = \langle a  : a^{3+r} = a^3 \rangle$, where $r$ is odd.
	\end{enumerate}
\end{theorem}
\begin{proof}
	
	Suppose that $\Gamma(S)$ is complete. First we claim that $S$ is of bounded exponent. Note that $o(a)$ is finite for all $a \in S$. Otherwise, there exists $a \in S$ such that $o(a)$ in infinite. Then $a^2 \nsim a^3$; a contradiction of the fact that $\Gamma(S)$ is complete. For each $x \in S$, there exists $a \in \mathcal M$ such that $x \in \langle a \rangle$. If  there exist  $x, y \in \mathcal M$ such that $\langle x \rangle  \ne \langle y \rangle$. Observe that $x \nsim y$; a contradiction. Thus, $S$ is a semigroup of exponent  $n$ for some $n \in \mathbb N$ and $S = \langle a \rangle$ for some $a \in \mathcal M$.
	Consider $S = M(m,r)$ for some $m, r \in \mathbb{N}$. On contrary, suppose that $S$ is not of the form given in (i), (ii)  and  (iii).
	Then either $S = M(3, r)$ such that $r$ is even or $S = M(m, r)$, where $m\geq 4$.
	
	If $S = M(3, r)$ such that $r$ is even, then clearly $3 + r - 1$ and $3 + r + 1$ are even. Since $a^{3+r} = a^3$ implies $a^{4+r} = a^4$ so that $\langle a^2 \rangle = \{a^2, a^4, \ldots, a^{2+r}\}$. Note that $a^3 \notin \langle a^2 \rangle$. If $a^2 \in \langle a^3 \rangle$, then $a^2 = a^{3k}$ for some $k\in \mathbb N$. Thus $m \leq 2$; a contradiction for $m = 3$. Consequently, $a^2 \notin \langle a^3 \rangle$.
	Let if possible $\langle a^2, a^3 \rangle = \langle a^t \rangle $ for some $a^t \in S$. We now show that no such $t \in \mathbb{N}$ exists. If $t = 1$, then $\langle a^2, a^3 \rangle  = \langle a \rangle$ so that
	$a = a^l$, where $l \ge 2$. Thus, $m = 1$; a contradiction. If $t \in \{2, 3 \}$, then either $a^2 \in \langle a^3 \rangle$ or $a^3 \in \langle a^2 \rangle$; again a contradiction. Thus, we have $\langle a^2,a^3 \rangle = \langle a^t \rangle$ such that $t > 3$. Since $a^2 \in \langle a^2,a^3 \rangle = \langle a^t \rangle$ so that $a^2 = (a^t)^k$ for some $k \in \mathbb N$. Consequently, $ m \leq 2$; a contradiction. Thus, $\langle a^2, a^3 \rangle$ is not a monogenic subsemigroup of $S$ implies $a^2 $ is not adjacent  to $a^3$ in $\Gamma(S)$ so that $\Gamma(S)$ is not complete which is a contradiction. \\
	We may now suppose $S = M(m, r)$, where $m \ge 4$. In this case, first note that $a, a^2, a^3, a^4$  all are distinct elements of $S$. Now, we show that $\langle a^2, a^3 \rangle$ is not a monogenic subsemigroup of $S$ so that $a^2$ and $a^3$ are not adjacent in $\Gamma(S)$, which is a contradiction of the fact $\Gamma(S)$ is complete. If possible, let  $\langle a^2, a^3 \rangle = \langle a^i \rangle$ for some $a^i \in S$. If $i = 1$, then $a \in \langle a^2, a^3 \rangle$. Thus, $a = a^t$, where $t \ge 5$ so that $m = 1$;  a contradiction. For $i = 2$, note that $a^3 \in \langle a^2 \rangle$ gives $a^3 = a^{2k}$ for some $k \ge 3$. Thus $ m \leq 3$; a contradiction. If $i \ge 3$, then $a^2 \in \langle a^i \rangle$, which implies that $a^2 = a^{ik}$ for some $k \in \mathbb N$. Thus $ m \leq 2$; again a contradiction.
	
	Conversely, suppose that $S$ is one of the form given in (i), (ii)  and  (iii). Thus, we have the following cases.
	
	\textbf{Case 1:} $S = M(1, r)$ i.e. $S = \{a, a^2, \ldots, a^r\}$. Since $S$ is a cyclic group, for any two distinct $x, y \in S$, note that $\langle x, y \rangle$ is a cyclic subgroup of $S$.
	Consequently, $\langle x, y \rangle$ is a monogenic subsemigroup of $S$ (cf. Lemma \ref{ch1-cyclic implies monogenic}) so that $\langle x, y \rangle  = \langle z \rangle $ for some $z \in S$. Thus,  $x \sim y$ in $\Gamma (S)$.
	Hence, $\Gamma(S)$ is complete.
	
	\textbf{Case 2:}  $S = M(2, r)$ i.e. $S = \{a, a^2, \ldots, a^{r+1}\}$ with $a^{2+r} = a^2$. Clearly, $\mathcal{K}_a = \{a^2, a^3, \ldots, a^{r+1}\}$. For $2 \leq i \leq r+1$,
	we have $a^i \in \langle a \rangle$ so that $\langle a^i, a \rangle = \langle a \rangle$. Thus $a \sim a^i$ in $\Gamma(S)$. Since $\mathcal{K}_a$ is a cyclic subgroup of $S$ and for any $a^i, a^j \in \mathcal{K}_a$,
	the subsemigroup $\langle a^i, a^j \rangle$ is monogenic in $S$ so that $a^i \sim a^j$ in $\Gamma(S)$. Thus, $\Gamma(S)$ is complete.
	
	\textbf{Case 3:} $S =M(3, r)$ such that $r$ is odd. Clearly,  $S = \{a, a^2, a^3, \ldots, a^{2+r} \}$ with $a^{3+r} = a^3 $.
	By the similar argument used in \textbf{Case 2}, note that for $ 2 \leq i \leq 2+r$, we have $a \sim a^i$  in $\Gamma(S)$. Since $3 + r$ is even implies $3 + r = 2k$ for some $k \in \mathbb{N}$.
	Thus, $a^3 = a^{3+r} = a^{2k} = (a^2)^k$ so that $a^3 \in \langle a^2 \rangle$. Consequently, $a^5, a^7, \cdots, a^{2+r} \in \langle a^2 \rangle$.
	For $i > 2$, note that $\langle a^2, a^i \rangle = \langle a^2 \rangle $ and it gives $a^2 \sim a^i$ in $\Gamma(S)$. Now, $\mathcal{K}_a = \{a^3, a^4, \ldots, a^{2+r} \}$ is a cyclic subgroup of $S$ and for any $a^i, a^j \in \mathcal{K}_a$ note that $\langle a^i, a^j \rangle$ is a monogenic subsemigroup of $S$. Thus $a^i \sim a^j$ in $\Gamma(S)$. Hence, $\Gamma(S)$ is complete.
\end{proof}

\begin{corollary}{\rm\cite[Theorem 9]{a.Ma2013}}\label{ch2-Gamma(G)-complete}
	Let $G$ be a finite group. Then the cyclic graph $\Gamma(G)$ is complete if and only if $G$ is a cyclic group.
\end{corollary}

\begin{theorem}\label{bipartite} Let $S$ be a semigroup. Then the following statements are equivalent:
	\begin{enumerate}[\rm (i)]
		\item  $\pi (S) \subseteq \{1,2\}$
		\item $\Gamma(S)$ is acyclic\index{acyclic graph}
		\item $\Gamma(S)$ is bipartite\index{bipartite}.
	\end{enumerate}	 
\end{theorem}

\begin{proof} 
	(i) $\Rightarrow$ (ii) Suppose $\pi(S) \subseteq \{1, 2\}$. Clearly, $S$ is of bounded exponent. Let if possible, there exists a cycle
	$C:$ $a_0 \sim a_1 \sim \cdots \sim a_k \sim a_0$, in $\Gamma(S)$. Then $C \subseteq S_f$ for some $f \in E(S)$. If none of the vertices of $C$ are idempotents, then $a_0, a_1 \in \langle z_1 \rangle$ for some $z_1 \in S$. Consequently, o$(z_1) \geq 3$; a contradiction. If one of the  vertex of  $C$ is idempotent, then note that there exist two non idempotent elements $a_i, a_j$ such that $a_i \sim a_j$. Thus, $a_i, a_j \in \langle z \rangle$ for some $z \in S$. Since $\langle z \rangle$  contains an idempotent element also, we get o$(z) \geq 3$; again a contradiction.
	
	\noindent (ii) $\Rightarrow$ (iii) Since $\Gamma(S)$ is acyclic graph so that it does not contain any cycle. By Theorem \ref{ch1-bipartite}, $\Gamma(S)$ is bipartite.
	
	\noindent (iii) $\Rightarrow$ (i) Suppose $\Gamma(S)$ is a bipartite graph.  By Theorem \ref{ch1-bipartite}, $\Gamma(S)$ does not contain any odd cycle. To prove $\pi(S) \subseteq \{1, 2\}$. Let if possible, there exist  $a \in S$ such that o$(a) \geq  3$. If $o(a)$ is infinite then clearly $a \sim a^2 \sim a^4 \sim a$ is a triangle in $\Gamma(S)$; a contradiction (see Theorem \ref{ch1-bipartite}). Now we assume that $o(a)$ is finite. Note that there exist $x, f \in S$ such that $f \in E(S) \cap \langle a \rangle$ and  $x \in \langle a \rangle \setminus \{a, f\}$. As a consequence, the vertices $x, a$ and $f$ forms a triangle; again a contradiction of the fact that $\Gamma(S)$ is bipartite.
\end{proof}

\begin{corollary}
	Let $S$ be a semigroup of bounded exponent. Then $\Gamma(S)$ is a tree\index{tree} if and only if $|E(S)| = 1$ and $\pi(S) \subseteq \{1, 2\}$.
\end{corollary}

\begin{corollary}
	Let $G$ be a group. Then the following statements are equivalent:
	\begin{enumerate}
		\item[\rm (i)]   exponent of $G$ is at most $2$. Moreover if $G$ is finite, then  $G \cong \mathbb Z_2 \times \mathbb Z_2 \times \cdots \times \mathbb Z_2$.
		\item[\rm (ii)] $\Gamma(G)$ is acyclic graph
		\item[\rm (iii)] $\Gamma(G)$ is bipartite
		\item[\rm (iv)] $\Gamma(G)$ is a tree
		\item[\rm (v)] $\Gamma(G)$ is a star graph\index{star graph}
	\end{enumerate}	 
\end{corollary}

\begin{theorem}
	Let $S$ be a semigroup of bounded exponent. Then $S$ is completely regular semigroup if and only if all the connected components of $\Gamma(S)$ forms a group.
\end{theorem}

\begin{proof}
	Suppose $S$ is completely regular semigroup. Then every $\mathcal H$-class of $S$ is a group (see Proposition \ref{ch1-completely-regular}). To prove that each connected component $\Gamma(S)$ forms a group, by  Corollary \ref{ch2-component-finite},  we show that $S_f = H_f$ for each $f \in E(S)$. Let $a \in H_f$. Then $a^n = f$  for some $n \in \mathbb N$ so that $a \in S_f$. On the other hand, suppose $a \in S_f$. If $a \in H_{f'}$ for some $f' \ne f \in E(S)$, then $a \in S_{f'}$;  a contradiction. Thus $H_f = S_f$.
	
	Conversely suppose that every connected component of $\Gamma(S)$ forms a group. To prove $S$ is completely regular, we show that every $\mathcal{H}$-class forms a group (see Proposition \ref{ch1-completely-regular}). Let $a \in S$. Then $a \in S_f$ for some $f \in E(S)$ as $S$ is of bounded exponent. We claim that $H_a = S_f$. Suppose $b \in S_f$. By Remark \ref{ch1-re-Green's-class}, $(b, f) \in \mathcal{H}$. Also, we have $(a, f) \in \mathcal{H}$ so that $(a, b) \in \mathcal{H}$. It follows that $S_f \subseteq H_a$. On the other hand, let $b \in H_a$. Then $a \in S_f$ and it implies that $b \in H_f$. Since $H_f$ contains an idempotent so that $H_f$ forms a group (see Corollary \ref{ch1-H-class}). It follows that $b^m =  f$ for some $m \in \mathbb N$. Hence, $H_a  = S_f$ for some $f \in E(S)$.
\end{proof}

\begin{theorem}\label{regular-graph} Let $S$ be a semigroup of bounded exponent. Then the cyclic graph $\Gamma(S)$ is  regular  if and only if  $|S_f| = |S_{f'}|$ for all $f, f' \in E(S)$ and for $f \in E(S)$ there exists $a \in S$ with one of the following holds:
	\begin{enumerate}
		\item[\rm (i)] $S_f = \langle a  : a^{1 + r} = a \rangle$
		\item[\rm (ii)] $ S_f =  \langle a  : a^{2+r} = a^2 \rangle$
		\item[\rm (iii)] $ S_f = \langle a  : a^{3+r} = a^3 \rangle$, where $r$ is odd.
	\end{enumerate}
\end{theorem}

\begin{proof}
	First suppose that $\Gamma(S)$ is regular and $f, f' \in E(S)$. Then deg$(f) = |S_f| -1 = |S_{f'}| -1 = {\rm deg}(f')$. Thus $|S_f| = |S_{f'}|$.  Let $f \in E(S)$ and $x, y \in S_f$. Consequently, $S_f$ forms a clique as $\Gamma(S)$ is a regular graph. Then $\langle x, y \rangle = \langle z \rangle$ for some $z \in S$. Observe  that $\langle z \rangle \subseteq S_f$. It follows that $S_f$ forms a subsemigroup of $S$. Since $\Gamma(S_f)$ is complete, by Theorem \ref{gamma-complete}, $S_f$ satisfies one of the given condition.
	
	Conversely, suppose  $x \in S_f$ and $y \in S_{f'}$ for some $f, f' \in E(S)$. By the given hypothesis, note that $\Gamma(S_f)$ and $\Gamma(S_{f'})$ are complete (see Theorem \ref{gamma-complete}). Then $|S_f| = |S_{f'}|$ follows that deg$(x) = {\rm deg}(y)$. Thus, $\Gamma(S)$ is regular.
\end{proof}

\begin{proposition}\label{ch3-isolated}
	An element $a$ of an arbitrary semigroup $S$ is an isolated vertex in $\Gamma(S)$ if and only if
	\begin{enumerate}[\rm (i)]
		\item $a$ is an idempotent in $S$.
		\item $ H_a = \{a\}$.
		\item $m_x = 1$ for each $x \in C(a)$.
	\end{enumerate}
\end{proposition}

\begin{proof}
	Let $a$ be an isolated vertex in $\Gamma(S)$. Clearly $a \in E(S)$.  Otherwise $a \sim a^2$. Consequently, $ H_a$ forms a group with the identity element $a$. Thus, every vertex of the cyclic graph induced by $H_a$ will be adjacent with $a$. Let if possible, $x \in H_a \setminus \{a\}$ then $x \sim a$; a contradiction. If $b \ne a \in S_a$ then $b^m = a$ for some $ m \in \mathbb N$. It follows that $b \sim a$; a contradiction. Thus, $S_a = \{a\}$ and $m_a  = 1$.
	
	Conversely suppose that for $a \in S$ satisfy (i), (ii) and (iii). Let if possible,  $a \sim x$ for some $x \in S$. Then $ \langle a, x \rangle = \langle b \rangle$ for some $b \in S$. Consequently, $a \sim b$ give $m_b = 1$. In view of Remark \ref{ch1-re-Green's-class}, $\langle b \rangle \subseteq H_a$ and so by (ii), $x = a$; a contradiction.
\end{proof}

\section{ The Clique Number of $\Gamma(S)$}\label{ch2-sec-clique}
In this section, we obtain the clique number of $\Gamma(S)$. The following lemma is useful in the sequel.

\begin{lemma}\label{ch1-gen-Ka^i}
	Let $S = M(m, r) = \langle a \rangle$ be a monogenic semigroup. Then  $\mathcal{K}_{a^i} = \langle a^{m+ g +i} \rangle$ for some $g$ such that $0 \leq g \leq r - 1$.
\end{lemma}

\begin{proof}
	For $m, r \in \mathbb N$, there exists $g \in \mathbb N$ such that $0 \leq g \leq r - 1$ and $m +g \equiv 0({\rm mod} \; r)$. Since $a^{m +g}$ is an idempotent element of $S$ and so is the identity element of $\mathcal{K}_{a^i}$. First we show that $\mathcal{K}_{a^i} = \langle a^{m +g +i} \rangle$. Clearly, $a^{m +g}x = x$ for all $x \in \mathcal{K}_{a^i}$. To prove that $\mathcal{K}_{a^i} \subseteq \langle a^{m +g +i} \rangle$, consider $x \in \mathcal{K}_{a^i}$. Then $x = (a^i)^t$ for some $t \in \mathbb N$. Consequently, we get $x = a^{m +g} a^{it} = (a^{m+g+i})^t \subseteq \langle a^{m +g +i} \rangle$. Now assume that $y \in  \langle a^{m +g +i} \rangle$ and so $y = (a^{m +g +i})^s$ for some $s \in \mathbb N$. It follows that $y = a^{m +g +si} = (a^i)^t a^{si} \in \langle a^i \rangle$. Thus $y \in \langle a^i \rangle \cap \mathcal{K}_a   = \mathcal{K}_{a^i}$. Therefore $\mathcal{K}_{a^i} = \langle a^{m +g +i} \rangle$.
\end{proof}

\begin{lemma}\label{ch2- adjacency-cond-more-than-one-index}
	For $m > 1$, let $S = M(m, r) =  \langle a \rangle$ be a monogenic semigroup such that $i < j$ and $i, j < m$.  Then the followings are equivalent:
	
	\begin{enumerate}[\rm (i)]
		\item $a^i \sim a^j$
		
		\item $a^j \in \langle a^i \rangle$
		
		\item $i \mid j$.
	\end{enumerate}
\end{lemma}

\begin{proof}
	(i) $\Rightarrow$ (ii) Firs suppose that $a^i \sim a^j$ in $\Gamma(S)$. Then $\langle a^i, a^j \rangle = \langle a^k \rangle $ for some $k \in \mathbb N$. It follows that $a^i  = a^{kt}$, $a^j = a^{ks}$ and $a^{k} = a^{ui + vj}$ for some $s, t \in \mathbb N$ and $u, v \in \mathbb{N}_0$. First note that $a^i \ne a^j$. If $a^i = a^j$ then $m \leq i$ which is not possible. From $a^k = a^{ui + vj}$, we get  $a^k = a^{(tu + sv)k}$. Also, observe that either $s > 1$ or $t > 1$. Otherwise, $a^i = a^j$; a contradiction. Now if $v = 0$ then (ii) holds. We may now suppose that $v \ne 0 $. If $ u = 0$, then $a^k \in \langle a^j \rangle$ implies $a^i \in \langle a^j \rangle$. Thus, $a^i = a^{jl}$ for some $ l \in \mathbb N$. Since $i < j$, we have $m \leq i$ but given that $i < m$. Therefore, $u \ne 0$. Consequently we have $u \ne 0$ and $v \ne 0$. Since $(su + tv)k > k$, we get $m \leq k$. Now consider $a^i = a^{kt}$.   If $kt \ne i$, then we get $m \leq i$; a contradiction. Thus, $kt = i$. If $t \ne 1$, then $k < i$ so $m \le k < i$; again a contradiction. It follows that $t = 1$ and so $i = k$. Therefore, $a^j \in \langle a^i \rangle$.
	
	\noindent (ii) $\Rightarrow$ (iii) Suppose $a^j \in \langle a^i \rangle$ so that $a^j = a^{it}$ for some $t \in \mathbb N$. If $j \ne it$ then $m \leq j$, which is not possible. Thus $j = it$ and so $i \mid j$.
	
	\noindent (iii) $\Rightarrow$ (i) Suppose $i \mid j$. Then $j = ki$ for some $k \in \mathbb N$. It follows that $\langle a^i, a^j \rangle = \langle a^i \rangle$ and hence $a^i \sim a^j$.
\end{proof}

\begin{proposition}\label{ch1-size-Ka^i}
	Let $S = M(m, r) = \langle a \rangle$ be a monogenic semigroup. Then $|\mathcal{K}_{a^i}| = \frac{r}{(i, \; r)}$, where $1 \leq i \leq m + r - 1$.
\end{proposition}

\begin{proof}
	In view of Lemma \ref{ch1-gen-Ka^i}, we have  $\mathcal{K}_{a^j} = \langle a^{m+ g +j} \rangle$ for some $g$, where $0 \leq g \leq r - 1$, $1 \leq j \leq m + r - 1$ and $m + g \equiv 0 ({\rm mod}\; r)$. We prove the result through the following cases:
	
	\noindent \textbf{Case 1:} $g = r - 1$. Then clearly $\mathcal{K}_a = \langle a^m \rangle$. Note that the map $\psi : \mathcal{K}_a \rightarrow \mathbb Z_r$ defined by $\psi(a^{m + j - 1}) = \bar j$, where $1 \leq j \leq r$ is a group isomorphism. Then $\psi(a^{m + i - 1}) = \bar i$.  By division algorithm, we get $x_i$ and $i'$ such that $i = rx_i + i'$ where $0 \leq i' < r$. Also note that $(i, r) = (i', r)$. By Lemma \ref{ch1-gen-Ka^i}, $\mathcal{K}_{a^i} = \langle a^{m+ r + i - 1} \rangle = \langle a^{m+ i - 1} \rangle = \langle a^{m+ rx_i + i' - 1} \rangle = \langle a^{m+ i' - 1} \rangle$. Therefore $|\mathcal{K}_{a^i}| = o(a^{m + i' -1}) = o(\bar{i'})$. It follows that $|\mathcal{K}_{a^i}| = \frac{r}{(i', \; r)} = \frac{r}{(i, \; r)}$.
	
	\noindent \textbf{Case 2:} $0 \leq g < r - 1$. Then clearly $\mathcal{K}_a = \langle a^{m +g +1} \rangle$. Note that the map $\phi : \mathcal{K}_a \rightarrow \mathbb Z_r$ defined by $\phi(a^{m +g +j}) = \bar j$ for $g \leq g + j \leq r -1$ and $\phi(a^{m +j}) = \overline{(r-g) +j}$ for $0 \leq j < g$ is a group isomorphism. Now if $g +i < r$ then $|\mathcal{K}_{a^i}| = o(a^{m +g + i}) = o(\bar i)$ in $\mathbb Z_r$. Thus $|\mathcal{K}_{a^i}| = \frac{r}{(i, \; r)}$. If $g +i \geq r$, then by division algorithm we get $g +i  = lr + i'$ where $0 \le i' < r$. We prove the result through the following subcases:
	
	\noindent \textit{Subcase 1:} $g < i'$. Then $i' = g + i''$ for some natural number $i''$. Since $\mathcal{K}_{a^i}  =  \langle a^{m +g + i} \rangle$ $= \langle a^{m + lr + i'} \rangle =  \langle a^{m +  i'} \rangle = \langle a^{m + g + i''} \rangle$. Therefore, we have 
	\begin{equation*} 
		\begin{split}
			|\mathcal{K}_{a^i}|  =  \;&  o(\overline{i''}) \; \text{ {\rm in}} \; \mathbb Z_r \; {\rm because} \;  g + i'' < r \\
			= \;& \frac{r}{(i'', \; r)} = \frac{r}{(i'- g, \; r)} = \frac{r}{(g + i - lr - g, \; r)}\\
			= \;& \frac{r}{(i - lr, \; r)} = \frac{r}{(i, \; r)}.\\
		\end{split}
	\end{equation*}
	
	\noindent \textit{Subcase 2:} $g \geq i'$. Then 
	\begin{equation*}
		\begin{split}
			|\mathcal{K}_{a^i}|  = \;&  o(a^{m + g + i}) = o(a^{m + lr + i'}) = o(a^{m + i'}) = o(\overline{r - g + i'}) \; {\rm in} \; \mathbb Z_r.
		\end{split}
	\end{equation*}
	
	It follows that 
	
	\begin{equation*}
		\begin{split}
			|\mathcal{K}_{a^i}|  &= \frac{r}{( r- g +i', \; r)} = \frac{r}{( r - g + g + i- lr, \; r)} =\frac{r}{(r(1-l) + i, \; r)} \\
			&= \frac{r}{(i, \; r)}.\\
		\end{split}
	\end{equation*}
\end{proof}


Now, in the following proposition, we obtain the clique number\index{clique! number} of $\Gamma(S)$, when $S$ is a monogenic semigroup . 

\begin{proposition}\label{ch2-clique-no-mono}
	Let $S = M(m, r) = \langle a \rangle$ be a monogenic semigroup. For $m > 1$, we have \[\omega(\Gamma(S)) = {\rm max} \{\mu_k : \; 1 \leq k < m \},\]
	$\mu_1 = 1 + r$ and for $ k \geq 2$, $\mu_k = 1 + \nu(k)  + \frac{r}{(r, \; k)}$, where $\nu(k)$ is the number of terms in the prime factorization of $k$.  
\end{proposition}

\begin{proof}
	Let $C$ be an arbitrary maximal clique of $\Gamma(S)$. To prove the result, we show that the  number of elements of index more than one in $C$ is $1 + \nu(k)$ and $|C \cap \mathcal{K}_a| = \frac{r}{(k, \; r)}$ for some $k$. First note that ${\rm N}[a] = S$ and so $a \in C$. Without loss of generality, we assume that $a, a^{i_1}, a^{i_2}, \ldots, a^{i_s}$ are in $C$ such that $1 < i_1 < i_2 < \cdots < i_s$ and index of each of these elements is more than one. We claim that $i_1$ is a prime. If $i_1$ is not a prime, then there exists a prime $p$ such that $p \mid i_1$. By Lemma \ref{ch2- adjacency-cond-more-than-one-index}, we get $a^p \sim a^j$ for all $j \in \{1, i_1, i_2, \ldots, i_s\}$. Let $a^q \in C \cap \mathcal{K}_a$. Then by the proof of Lemma \ref{ch2- adjacency-cond-more-than-one-index}((i) $\Rightarrow$ (ii)), we get $a^q \in \langle a^{i_s} \rangle$. Again by Lemma \ref{ch2- adjacency-cond-more-than-one-index}, we get $\langle a^q \rangle \subseteq \langle a^{p} \rangle$. It follows that $\langle a^{q}, a^p \rangle =  \langle a^{p} \rangle$ and so $a^q \sim a^p$. Consequently, $C \cup \{a^p\}$ forms a clique; a contradiction to the maximality of $C$. Thus, $i_1$ is a prime and we write $i_1 = p_1$. Now $a^{p_1} \sim a^{i_2}$ implies $p_1 \mid i_2$ (see Lemma \ref{ch2- adjacency-cond-more-than-one-index}). We get $i_2 = p_1t$ for some $t \in \mathbb N \setminus \{1\}$. Note that $t$  is a prime. Otherwise, there exists a prime $p'$ such that $p' \mid t$. Since $p_1 p' \mid p_1t = i_2$, by the similar argument as used above, $C \cup \{a^{p_1 p'}\}$ forms a clique and again we get a contradiction. As a result, $t$ is a prime and we write $t = p_2$ so that $i_2 = p_1 p_2$. On continuing similar process, we get $i_s = p_1 p_2 \ldots p_s$.
	
	Thus, there are $1 + \nu(i_s)$ elements of index more than one in $C$. Now we count the number of elements of index one in $C$. We show that $C \cap \mathcal{K}_a = \mathcal{K}_{a^{i_s}}$ so that $|C \cap \mathcal{K}_a| = \frac{r}{(i_s, \; r)}$ (see Proposition \ref{ch1-size-Ka^i}).

	To prove  $C \cap \mathcal{K}_a = \mathcal{K}_{a^{i_s}}$, consider $a^t \in \mathcal{K}_{a^{i_s}}$. Since $\mathcal{K}_{a^{i_s}} = \mathcal{K}_a \cap \langle a^{i_s} \rangle$ (cf. Remark \ref{ch1-Ka^i}), we get $a^t \in \langle a^{i_s} \rangle \subseteq  \langle a^j \rangle$ for all $j$, where $j \in \{1, i_1, i_2, \ldots, i_s\}$. Therefore, $a^t \sim a^j$. Since the subgraph induced by $\mathcal{K}_a$ is complete and  $a^t \in \mathcal{K}_a$, we get  $a^t$ is adjacent with all the elements of $C$ whose index is one. If $a^t \notin C$, then $C \cup \{a^t\}$ forms a clique; a contradiction to the maximality of $C$. Consequently $a^t \in C \cap \mathcal{K}_a$ so that $\mathcal{K}_{a^{i_s}} \subseteq \mathcal{K}_a \cap C$. Let $a^l \in \mathcal{K}_a \cap C$. Then $a^l \sim a^{i_s}$. By the proof  of Lemma \ref{ch2- adjacency-cond-more-than-one-index}((i) $\Rightarrow$ (ii)), we get $a^l \in \langle a^{i_s} \rangle$. Thus, $a^l \in \mathcal{K}_a \cap \langle a^{i_s} \rangle = \mathcal{K}_{a^{i_s}}$ so that $\mathcal{K}_a \cap \langle a^{i_s} \rangle = \mathcal{K}_{a^{i_s}}$.
	
	Now, for each $k$ such that $1 \leq k < m$, we provide a clique of size $\mu_k$. For $k = 1$, $\mathcal{K}_a \cup \{a\}$ forms a clique of size $1 + r$. For $k > 1$, we have $k = p_1 p_2 \ldots p_{s'}$. Then by Lemma \ref{ch2- adjacency-cond-more-than-one-index}, note that $\{a, a^{p_1}, a^{p_1 p_2}, \ldots, a^k \} \cup \mathcal{K}_{a^k}$ forms a clique. Because if $x \in \mathcal{K}_{a^k} = \langle a^k \rangle \cap \mathcal{K}_a$, then $x \in \langle a^k \rangle \subseteq \langle a^j \rangle$ for all $j \in \{p_1, p_1p_2,\ldots, k\}$. Consequently, $x \sim a^j$ for all $j$. This completes our proof.
\end{proof}

The following lemma will be useful in the sequel.

\begin{lemma}[{\rm\cite[Lemma 32]{a.Aalipour2017}}]\label{ch1-epg-cyclic-subgroup}
	Let $G$ be a group and suppose $x, y, z \in G$ such that $\langle x, y \rangle$, $\langle y, z \rangle$ and $\langle x, z \rangle$ are the cyclic subgroup of $G$. Then $\langle x, y, z \rangle$ is cyclic subgroup of $G$.  
\end{lemma}

\begin{lemma}\label{ch1-exponent}
	Let $S$ be a semigroup with exponent $n$. Then for $x \in S$, we have
	\begin{enumerate}[\rm (i)]
		\item $o(x) \leq 2n$ for all $x \in S$.
		
		\item the subsemigroup $\langle x \rangle$ is contained in some maximal monogenic subsemigroup of $S$.
	\end{enumerate}
\end{lemma}

\begin{proof}
	(i) Since $x^n = f$ for some $f \in E(S)$, we have $\langle x \rangle = M(m, r)$ for some $m, r \in \mathbb N$. There exists $g$ with $0 \leq g < r$ and $m +g \equiv 0 ({\rm mod} \; r)$ such that $x^{m + g} = f$. Clearly, $n \geq m$ as $x^n$ is the idempotent element of $\langle x \rangle$. Let if possible $n < m +g$. For $m \leq i \ne j \leq m + r- 1$, we have $x^i \ne x^j$ in $\langle x \rangle$. Therefore, $x^n \ne x^{m + g}$, which is not true because $x^n = x^{m + g} = f$. Thus, $n \geq m + g$. Also $r \mid m + g$ and $m \leq n$ gives $r \leq n$. It follows that $o(x) \leq m + r \leq 2n$. 
	
	\noindent (ii) If $\langle x \rangle$ is a maximal monogenic subsemigroup, then the result holds. Otherwise, $\langle x \rangle \subsetneq \langle x_1 \rangle$. If $\langle x_1 \rangle$ is maximal monogenic then this completes our proof. By (i), since $o(x) \leq 2n$, we get a finite chain such that $\langle x \rangle \subsetneq \langle x_1 \rangle \cdots \subsetneq \langle x_{k-1} \rangle \subsetneq \langle x_k \rangle$, where $k \leq 2n$ and $\langle x_k \rangle$ is a maximal monogenic subsemigroup of $S$. This complete the proof.
\end{proof}

\begin{lemma}\label{ch1-index-a^i}
	Let $a^i \in \langle a \rangle  = M(m, r)$. Then
	\begin{enumerate}[\rm (i)]
		\item $m_{a^i} = 1$ if and only if $i \geq m$.
		\item $m_{a^i} = 2$ if and only if $\left \lceil \frac{m}{2} \right\rceil \leq i \leq m - 1$.
	\end{enumerate}
\end{lemma}

\begin{proof}
	(i) Let $a^i \in \langle a \rangle  = M(m, r)$. If $i \geq m$, then $a^i \in \mathcal{K}_a$. Since $\mathcal{K}_a$ is a finite group of order $r$ so $(a^i)^{r + 1} = a^i$. Consequently, we get $m_{a^i} = 1$. To prove the converse part, we show that $m_{a^i} > 1$ for all $i < m$. On contrary, we assume that $i < m$ such that $m_{a^i} = 1$. Clearly, $m > 1$. Then $(a^i)^t = a^i$ for some $t > 1$ gives $m \leq i$; a contradiction.
	
	\smallskip
	\noindent
	(ii) For $\left \lceil \frac{m}{2} \right\rceil \leq i \leq m - 1$, we have $(a^i)^2 \in \mathcal{K}_a$. Therefore, $((a^i)^2)^{r + 1} = (a^i)^2$ gives $m_{a^i} \leq 2$. If $m_{a^i} = 1$, then we must have $m \leq i$; a contradiction. It follows that $m_{a^i} = 2$. For the converse part, we assume that $m_{a^i} = 2$. By (i), note that $i < m$. Since $m_{a^i} = 2$ then we get  $m \leq 2i$. If $1 \leq i < \left \lceil \frac{m}{2} \right\rceil$, then $2i < m$ which is not possible. Thus, we have $\left \lceil \frac{m}{2} \right\rceil \leq i \leq m - 1$.  
\end{proof}

\begin{lemma}\label{ch1-o(a^i)-in-o(a)}
	Let $S = \langle a \rangle$ be a monogenic semigroup and  $x \ne a \in  \langle a \rangle$. If $m_a > 1$, then $m_x \leq m_a$ and $o(x) < o(a)$. 
\end{lemma}

\begin{proof}
	Let $x \ne a \in  \langle a \rangle$. Then $x = a^i$ for some $i > 1$. Clearly, $\langle a^i \rangle \subsetneq \langle a \rangle$. Otherwise, $\langle a^i \rangle = \langle a \rangle$ gives $a = (a^i)^k$ gives $m_a = 1$; a contradiction. For $i \geq m_a$, we have $a^i \in \mathcal{K}_a$ and so $m_x = 1 \leq m_a$ (cf. Lemma \ref{ch1-index-a^i}).  If $i < m_a$, then $(a^i)^{m_a} \in \mathcal K_a \cap \langle a^i \rangle = \mathcal{K}_{a^i}$ implies $m_x \leq m_a$.
\end{proof}

\begin{proposition}\label{ch2-maximal-cliq-monoge}
	Let $S$ be a semigroup with exponent $n$ and  $C$ be a  maximal clique in $\Gamma(S)$. Then $C \subseteq \langle a \rangle$ for some $a \in S$.
\end{proposition}

\begin{proof}
	Let $x \in C$. Then $x^n = f$ for some $f \in E(S)$. In view of Corollary \ref{ch2-component-finite}, $x \in S_f$ and $S_f$ is a connected component of $\Gamma(S)$. Since $C$ is a clique, we get $C \subseteq S_f$. Now consider $\mathsf{M} = {\rm max}\{m_x : x \in C \}$. We prove our result through the following cases:
	
	\noindent \textbf{Case 1:} $\mathsf{M} = 1$. First we prove that $\langle C \rangle$ is a subgroup of $S$. For that, suppose $x \in C \subseteq S_f$. Then $x^n = f$ for some $n \in \mathbb N$. Since $m_x = 1$, we get $\langle x \rangle$ is a cyclic subgroup of $S$. It follows that $xf = fx = x$. Further note that for any $x, y \in C$, we have $xy = yx$. Consequently, for $a \in \langle C \rangle$, we have $a = c_1^{k_1} c_{2}^{k_2} \ldots c_n^{k_n}$, where $c_i \in C$ and $k_i \in \mathbb N$. Observe that $af = a$ so that $\langle C \rangle$ forms a monoid with the identity element $f$. Since  $a = c_1^{k_1} c_{2}^{k_2} \ldots c_n^{k_n}$, we have $ab = ba = f$, where $b = (c_n^{k_n})^{-1}  \ldots (c_1^{k_1})^{-1}$. Thus, $\langle C \rangle$ is a subgroup of $S$. Now we show that $ C$ is a cyclic subgroup of $S$ and let $x, y, z \in C$. By Lemma \ref{ch1-epg-cyclic-subgroup}, $\langle x, y, z \rangle$ is a cyclic subgroup of $\langle C \rangle$. Consequently, $x^iy^j \sim z$ for each $i, j \in \mathbb N$. It follows that $C \cup \langle x, y \rangle$ is a clique of $\Gamma(S)$. Since $C$ is a maximal clique, we must have $\langle x, y\rangle \subseteq C$. Therefore, $a \in C$ so that $\langle C \rangle \subseteq C$ gives $\langle C \rangle = C$. Thus $C$ is a subgroup of $S$. In view of Lemma \ref{ch1-exponent}, $o(x) \leq 2n$ for all $x \in C$. Choose $x \in C$ such that $o(x) \geq o(y)$ for all $y \in C$. In order to prove $C \subseteq \langle x \rangle$, let $y \in C$. Then $\langle x, y \rangle = \langle z \rangle$ for some $z \in C$ implies $y \in \langle z \rangle = \langle x \rangle$. Thus the result holds.

	\noindent \textbf{Case 2:} $\mathsf{M} > 1$. By Lemma \ref{ch1-exponent}, $o(x) \leq 2n$ for all $x \in C$. Now choose $x \in C$ with $m_x = \mathsf{M}$ and $o(x) \geq o(y)$ for all $y \in C$ such that $m_y = \mathsf{M}$. We show that $C \subseteq \langle x \rangle$. Let $y \in C$. Then $\langle x, y \rangle = \langle z \rangle$. It follows that $x = z^i$, $y = z^j$ and $z = x^u y^v$ for some $i, j \in \mathbb N$ and $u, v \in \mathbb N_0$. If either $v =0$ or $i = 1$, then observe that $y \in \langle z \rangle \subseteq \langle x \rangle $. Therefore, $C \subseteq \langle x \rangle$. We may now suppose that $v\ne 0$ and $i > 1$. If $u \ne 0$, then $z = z^{ui + vj}$ gives $m_z = 1$ and so $m_x = 1$; a contradiction because $m_x = \mathsf{M} > 1$. Consequently, we get $u  = 0$ and so $x \in \langle y \rangle$. By Lemma \ref{ch1-o(a^i)-in-o(a)}, $m_x \leq m_y$. Thus, $m_x = m_y = \mathsf{M}$ gives $o(x) \geq o(y)$. Since $x \in \langle y \rangle$, we obtain $x = y^l$ for some $l \in \mathbb N$. If $l > 1$, then $o(x) < o(y)$ (cf. Lemma \ref{ch1-o(a^i)-in-o(a)}) which is not possible. Thus, $x = y$ and hence $C \subseteq \langle x \rangle$. 
\end{proof}

\begin{theorem}\label{ch2-cliq-finiite-expo}\index{clique! number}
	Let  $S$ be a semigroup with exponent $n$. Then\\ $\omega(\Gamma(S)) = {\rm max} \left( \{r_{a} :  \; m_a = 1 \; {\rm and} \; a \in \mathcal{M} \} \cup \{\mu_k^a :  \; m_a > 1, \; 1 \leq k < m_{a} \; {\rm and} \; a \in \mathcal{M}\} \right),$
	$\mu_1^a = 1 + r_{a}$ and for $ k \geq 2$, $\mu_k^a = 1 + \nu(k)  + \frac{r_{a}}{(r_{a}, \; k)}$, where $\nu(k)$ is the number of terms in the prime factorization of $k$.  
\end{theorem}

\begin{proof}
	Let $C$ be a clique of maximum size in $\Gamma(S)$. Consider the sets 
	\[A = \{r_{a} : \; a \in \mathcal{M}, \; m_a = 1 \} \; {\rm and}  \; B = \{\mu_k^a : \; a \in \mathcal{M}, \; m_a > 1, \; 1 \leq k < m_{a}\}.\]
	We claim that $|C| \in A \cup B$. By Proposition \ref{ch2-maximal-cliq-monoge}, $C \subseteq \langle a' \rangle$ for some $a' \in S$ and $a' \in \langle a \rangle$ for some $a \in \mathcal{M}$ (cf. Lemma \ref{ch1-exponent}). Then $C \subseteq \langle a \rangle$. If $m_a = 1$, then $\langle a \rangle$ is a cyclic subgroup of $S$ and so $\Gamma(\langle a \rangle)$ is complete ( see Corollary \ref{ch2-Gamma(G)-complete}). Since $C$ is a clique of maximum size in $\Gamma(S)$ so it is of maximum size  in $\Gamma(\langle a \rangle)$ also. It follows that $C = \Gamma(\langle a \rangle)$. Consequently, we get $|C| = |\langle a \rangle| = r_{a} \in A$. Now let $m_a > 1$. Then $C$ is again a clique of maximum size  in $\Gamma(\langle a \rangle)$. By Proposition \ref{ch2-clique-no-mono}, $|C| = \mu_k^a \in B$ for some $k$, where $1 \leq k < m_{a}$. Next we provide a clique of size $t$ for each $t \in A \cup B$. If $t \in A$, then there exists $a \in \mathcal{M}$ such that $t = r_{a}$ and $m_{a} = 1$. Thus, $\langle a \rangle$ is a cyclic subgroup of $S$ and the subgraph induced by $\langle a \rangle$ is complete ( cf. Corollary \ref{ch2-Gamma(G)-complete}). We get a clique of  size $V(\Gamma(\langle a \rangle)) = o(a) = r_{a} = t$. If $t \in B$, then $t = \mu_k^a$ of some $a \in \mathcal{M}$ such that $ m_a > 1$ and $1 \leq k < m_{a}$. Since the prime factorization of $k$ is $p_1 p_2 \ldots p_s$ and by the proof of Proposition \ref{ch2-clique-no-mono}, the set $\{a, a^{p_1}, a^{p_1p_2}, \ldots, a^k\} \cup \mathcal{K}_{a^k}$ forms a clique of size $t = \mu_{k}^a$. This completes the proof.
\end{proof}

In view of {\cite[Theorem 2.5]{a.Dalal2020chromatic}}, we have the following corollary.

\begin{corollary}\label{ch2-clique-most-countable}
	The clique number of the cyclic graph  of any semigroup  is at most countable.
\end{corollary}

\begin{theorem}\label{ch2-clique-unbd-expo}
	Let $S$ be a semigroup of unbounded exponent. Then $\omega(\Gamma(S))$ is countably infinite.
\end{theorem}

\begin{proof}
	In view of Corollary \ref{ch2-clique-most-countable}, to prove the result, we show that for $k \in \mathbb N$ there exists a clique of size $\left\lfloor \log_2 k \right \rfloor + 1$. We claim that: there exists $a \in S$ such that $a, a^2, \ldots, a^k$ are non idempotent elements of $S$. Let, if possible, there exists $i_a \leq k$ such that  $a^{i_a} = f$ for some $f \in E(S)$. Now choose $n = k!$. Note that $a^n = (a^{i_a})^{23\cdots (i_a -1)(i_a + 1) \cdots k} = f$. Thus, $S$ is of bounded exponent; a contradiction. This proves the claim. By Lemma \ref{ch2- adjacency-cond-more-than-one-index}, note that the sets $\{a, a^2, a^4, \ldots, a^{2^{\left\lfloor \log_2 k \right \rfloor}}\}$ forms a clique of size $\left\lfloor \log_2 k \right \rfloor + 1$. This completes our proof.
\end{proof}

In view of Theorem \ref{ch2-clique-unbd-expo}, we have the following corollary.

\begin{corollary}\label{ch2-chro-unbd-expo}
Let $S$ be a semigroup of unbounded exponent. Then $\chi(\Gamma(S))$ is countably infinite.
\end{corollary}

\section{The Independence Number of $\Gamma(S)$}\label{ch2-sec-inde}
In this section, we investigate the independence number $\alpha(\Gamma(S))$ of $\Gamma(S)$. First, we obtain $\alpha(\Gamma(S))$ for a monogenic semigroup $S$ in the following theorem.

\begin{theorem}\label{ch2-independence-monogenic}
	Let $S = \langle a \rangle$ be a monogenic semigroup. Then the independence number\index{independence number} of $\Gamma(S)$ is given below:
	\[\alpha(\Gamma(S)) = \left\{ \begin{array}{ll}
		\infty & \; {\rm if} \;  S \; {\rm  is \; infinite};\\
		
		1 & \; {\rm if} \;  S \; {\rm  is \;  finite \; and} \; m_a = 1;\\
		
		\left\lfloor \frac{m_a}{2} \right\rfloor + 1 & \; {\rm if} \; m_a > 1, \;  (i, r_a) > 1 \; {\rm for \; all} \;  i, \; {\rm  where} \;  \left\lceil \frac{m_a}{2} \right\rceil \leq i \leq m_a-1; 	\vspace{0.1cm}\\

		\left\lfloor \frac{m_a}{2} \right\rfloor & \; {\rm if} \; m_a > 1, \;  (i, r_a) = 1\; {\rm for \; some}\; i, {\rm where} \;  \left\lceil \frac{m_a}{2} \right\rceil \leq i \leq m_a -1.
	\end{array} \right. \]
\end{theorem}

\begin{proof}
	Suppose that $S$ is an infinite semigroup. We first claim that for $i < j$, we have $a^i \sim a^j$ in $\Gamma(S)$ if and only if $i \mid j$. If $i \mid j$, then clearly $a^j \in \langle a^i \rangle$ so that $a^i \sim a^j$ in $\Gamma(S)$. On the other hand, if  $a^i \sim a^j$, then $\langle a^i, a^j  \rangle = \langle a^k \rangle$ for some $a^k \in S$. Thus, $a^i = a^{tk}, \; a^j = a^{t'k}$  and $a^k = a^{ui} + a^{vj}$ for some $t, t' \in \mathbb N$ and $u, v \in \mathbb N_0$. Therefore, $a^k = a^{ui + vj} = a^{(tu + t'v)k}$. If $tu + t'v \ne 1$, then $m_a \leq k$ which is not possible as $S$ is an infinite semigroup. Thus, $tu + t'v = 1$. Further note that both $t, t'$ can not be $1$. Otherwise, $a^i = a^j$ which is not possible. It follows that either $t > 1$ or $t' > 1$. If $t > 1$ then $u = 0$ as $tu + t'v = 1$.  We get $a^i \in \langle a^j \rangle$ so that $a^i = a^{jl}$ for some $l \in \mathbb N$. For $i \ne jl$, we have $m_a \leq i$; a contradiction. Consequently, $i = jl$ and so $j \mid i$ which is not possible as $i < j$. If $t' > 1$, then by the similar argument used above, we get $j = il'$ for some $l' \in \mathbb N$. Thus, $i \mid j$. Consequently, the set $\{a^p : \; p \; {\rm is \; a \; prime} \}$ is an independent in $\Gamma(S)$ so that $\alpha(\Gamma(S)) = \infty$. 
	
	Now we prove our result, when $S$ is finite. If $m_a = 1$, then $S = \langle a \rangle$ is a cyclic group and therefore by Corollary \ref{ch2-Gamma(G)-complete}, $\Gamma(S)$ is complete. It follows that $\alpha (\Gamma(S)) = 1$. We now assume that $m_a > 1$. Consider the set 
	\[\mathcal I = \{a^i : \;  \left\lceil \frac{m_a}{2} \right \rceil\leq i \leq m_a-1 \}.\]
	
	By Lemma \ref{ch2- adjacency-cond-more-than-one-index}, $\mathcal{I}$ is an  independent set of size $\left\lfloor \frac{m_a}{2} \right\rfloor$. Further, we split our proof in two cases:
	
	\noindent \textbf{Case 1:} $(i, r_a) > 1 \; {\rm for \; all} \;  i, \; {\rm  where} \;  \left\lceil \frac{m_a}{2} \right\rceil \leq i \leq m_a-1$. Then $|K_{a^i}| < r_a$ (cf. Proposition \ref{ch1-size-Ka^i}).  Since $\mathcal{K}_a = \langle a^{m_a +g} \rangle$ for some $g$, where  $0 \leq g \leq r_a - 1$ and $m_a + g \equiv 1 ({\rm mod} \; r_a)$. Note that  $\mathcal{I}\cup \{a^{m_a +g}\}$ is an independent set. If $a^{m_a +g} \sim a^i$ for some $ i, \; {\rm  where} \;  \left\lceil \frac{m_a}{2} \right\rceil \leq i \leq m_a-1$, then $a^{m_a + g} \in \langle a^i \rangle$ (see proof of Lemma \ref{ch2- adjacency-cond-more-than-one-index} (i) $\Rightarrow$ (ii)). Thus, $\mathcal{K}_a \subseteq \mathcal{K}_{a^i}$ and so $\mathcal{K}_a = \mathcal{K}_{a^i}$, which is a contradiction of $|\mathcal{K}_{a^i}| < r_a$. To prove our result, in this case, we show that if $\mathcal{I'}$ is an arbitrary independent set in $\Gamma(S)$ then $|\mathcal{I'}| \leq \left\lfloor \frac{m_a}{2} \right\rfloor + 1 $. Since the subgraph induced by $\mathcal{K}_a$ is complete, we get $|\mathcal{K}_a \cap \mathcal{I'}| \leq 1$. Without loss of generality, consider the set \[\mathcal{I'} \cap (S \setminus \mathcal{K}_a) = \{a^{i_1}, a^{i_2}, \ldots, a^{i_t}, a^{i_{t + 1}}, \ldots, a^{i_l}\},\] where $i_1 < i_2 < \cdots < i_t < \frac{m_a}{2}$ and $\frac{m_a}{2} \leq i_{t + 1} < i_{t + 2} < \cdots < i_l < m_a$. For each $i_s \in \{i_1, i_2, \ldots, i_t\}$, we have $2 i_s < m_a$. Now choose the smallest natural number $\alpha_s$ such that $i_s 2^{\alpha_s + 1} \geq m_a$. Then $\frac{m_a}{2} \leq i_s 2^{\alpha_s} < m_a$. We claim that if $i_{s_1} \ne i_{s_2}$ then $i_{s_1} 2^{\alpha_{s_1}} \ne i_{s_2} 2^{\alpha_{s_2}}$. If $i_{s_1} 2^{\alpha_{s_1}} = i_{s_2} 2^{\alpha_{s_2}}$, then clearly $\alpha_{s_1} \ne \alpha_{s_2}$. Without loss of generality, we assume that $\alpha_{s_1} > \alpha_{s_2}$. Thus, $i_{s_1} (2^{\alpha_{s_1} - \alpha_{s_2}}) = i_{s_2}$ implies $i_{s_1} \mid i_{s_2}$. By Lemma \ref{ch2- adjacency-cond-more-than-one-index}, $a^{i_{s_1}} \sim a^{i_{s_2}}$; a contradiction of the fact that $\mathcal{I'}$ is an independent set. Moreover, for each $i_s \in \{i_1, i_2, \ldots, i_t\}$, $a^{i_s} \sim a^{i_s 2^{\alpha_s}}$ and $ a^{i_s 2^{\alpha_s}} \in \mathcal{I}$ but $ a^{i_s 2^{\alpha_s}}$ can not be in $\mathcal{I'}$. Thus, we have $|\mathcal{I'} \cap (S \setminus \mathcal{K}_a)| \leq t + \left\lfloor \frac{m_a}{2} \right\rfloor - t = \left\lfloor \frac{m_a}{2} \right\rfloor$.

	\noindent \textbf{Case 2:} $(i, r_a) = 1 \; {\rm for \; some} \;  i, \; {\rm  where} \;  \left\lceil \frac{m_a}{2} \right\rceil \leq i \leq m_a-1$. Then $|\mathcal{K}_{a^i}|  = r_a = |\mathcal{K}_a|$ and so $\mathcal{K}_{a^i} = \mathcal{K}_a$. Thus, $\mathcal{K}_a \subseteq \langle a^i \rangle$. It follows that $a^i \sim x$ for all $x \in \mathcal{K}_a$. Now let $j < \frac{m_a}{2}$. Then by the similar argument used in \textbf{Case 1}, we get $a^j \sim a^{j 2^{\alpha_j}}$, where $ a^{j 2^{\alpha_j}} \in \mathcal{I}$. Consequently, $\mathcal{I}$ is a maximal independent set. Now to prove $\mathcal{I}$ is an independent set of maximum size, we assume that $\mathcal{I''}$ is an independent set different from $\mathcal{I}$. Since the subgraph induced by $\mathcal{K}_a$ is complete, we get $|\mathcal{I''} \cap \mathcal{K}_a| \leq 1$. Also,  by the similar argument used in \textbf{Case 1}, we get $|\mathcal{I''} \cap (S \setminus \mathcal{K}_a)| \leq \left\lfloor \frac
	{m_a}{2} \right\rfloor$. If $|\mathcal{I''} \cap \mathcal{K}_a| = 0$, then $|\mathcal{I''}| \leq \left\lfloor \frac{m_a}{2} \right\rfloor$. If $|\mathcal{I''} \cap \mathcal{K}_a| = 1$, then there exists $a^j \in \mathcal{I''} \cap \mathcal{K}_a$. Since $\mathcal{K}_a = \mathcal{K}_{a^i}$, we get $a^j \in \mathcal{K}_{a^i} = \langle a^i \rangle \cap \mathcal{K}_a$. It follows that $a^j \sim a^i$ and so $a^i \in \mathcal{I}$ but $a^i \notin \mathcal{I''}$. Again by the similar argument used in \textbf{Case 1}, we get $|\mathcal{I''}| \leq \left\lfloor \frac{m_a}{2} \right\rfloor$. Thus, $\mathcal{I}$ becomes an independent set of maximum size $\lfloor \frac{m_a}{2} \rfloor$. This completes our proof.
\end{proof}

The following lemma will be useful in the sequel.

\begin{lemma} [{\cite[Lemma 2.10]{a.Dalal2020chromatic}}] \label{index-one}
	For $k > 1$, the set $I_k = \{x \in S \; : \; m_x = k \}$ is independent in $\Gamma(S)$.
\end{lemma}

Now we determine a lower and upper bound of $\alpha(\Gamma(S))$, when $S$ is a semigroup of exponent $n$. Consider  a relation $\tau$ on $S$ defined by $x \; \tau \; y$ if and only if $\langle x \rangle = \langle y \rangle$. Clearly $\tau$ is an equivalence relation. Let $X$ be a complete set of distinct representative elements for $\tau$. Now, let $I_2 = \{x \in S : \; m_x = 2 \}$ and
$J_2 = \{a \in \overline{\mathcal M} \cap X : \; a \notin \langle x \rangle \; {\rm for \; any} \; x \in I_2 \}$.

\begin{theorem}\label{ch2-inde-bound}
	Let $S$ be a  semigroup with exponent $n$ and $\mathcal{M}$ is finite. Then
	\[|I_2| + |J_2| \leq \alpha(\Gamma(S)) \leq |J_2| + \mathop{\sum}_{a \in \mathcal{M}} \left\lfloor \frac{m_{a}}{2} \right\rfloor. \]  
\end{theorem}

\begin{proof}
	To find a lower bound of $\alpha(\Gamma(S))$, we show that $I_2 \cup J_2$ is an independent set of $\Gamma(S)$. By Lemma \ref{index-one}, $I_2$ is an independent set. If $a, b \in J_2$ such that $a \sim b$, then $\langle a, b \rangle = \langle z \rangle$ for some $z \in S$. Since $o(z) \leq 2n$ (see Lemma \ref{ch1-exponent})  and $m_a = m_b = 1$, we get $\langle a, b \rangle = \langle z \rangle$ is a cyclic subgroup of $S$. The maximality of $\langle a \rangle$  and $\langle b \rangle$ follows that $\langle z \rangle = \langle a \rangle = \langle b \rangle$, which is not possible. Consequently, $J_2$ is an independent set. Further, to show $I_2 \cup J_2$ is an independent set, let $x \in I_2$ and $y \in J_2$ such that $x \sim y$. By Lemma \ref{ch1-index-a^i} and by the proof of Lemma \ref{ch2- adjacency-cond-more-than-one-index}, we have $y \in \langle x \rangle$; a contradiction of $y \in J_2$. Thus, $I_2 \cup J_2$ is an independent set. Since $I_2 \cap J_2 = \varnothing$, we get $\alpha (\Gamma(S)) \geq |I_2| + |J_2|$.
	
	Now we obtain an upper bound for $\alpha (\Gamma(S))$. Note that the sets 
	\[A = \left\{a \in \mathcal{M} : \; m_a > 1, \; (i, r_a) > 1 \; {\rm for \; all} \; i, \; \left\lceil \frac{m_a}{2} \right\rceil  \leq i \leq m_a -1 \right\},\]
	\[B = \left\{a \in \mathcal{M} : \; m_a > 1, \; (i, r_a) = 1 \; {\rm for \; some} \; i, \; \left\lceil \frac{m_a}{2} \right\rceil  \leq i \leq m_a -1 \right\},\]
	and $C = \{a \in \mathcal{M} : \; m_a = 1 \}$ forms a partition of $\mathcal{M}$. Since $S$ is of exponent $n$, by Lemma \ref{ch1-exponent}, note that $S = \left(\displaystyle\mathop{\cup}_{a \in A}\langle a \rangle\right) \cup \left(\displaystyle\mathop{\cup}_{b \in B}\langle b \rangle\right) \cup \left(\displaystyle\mathop{\cup}_{c \in C}\langle c \rangle\right)$. Let $\mathcal{I}$ be any independent set in $\Gamma(S)$. We assume that $x_1, x_2, \ldots, x_l$,\\ $y_1, y_2, \ldots, y_m, z_1, z_2, \ldots, z_n$ are the elements of index one belongs to $\mathcal{I}$ such that

	\begin{itemize}
		\item for each $i$, where $1 \leq i \leq l$, $x_i \in  \langle a \rangle$ for some $a \in A$ and $x_i \notin \langle b \rangle$ for any $b \in B$. \hfill $(2.1)$
		
		\item for each $j$, where $1 \leq j \leq m$, $y_j \in  \langle b \rangle$ for some $b \in B$. \hfill $(2.2)$
		
		\item for each $k$, where $1 \leq k \leq n$, $z_k \in  \langle c \rangle$ for some $c \in C$ and $z_k \notin \langle b \rangle$ for any $b \in B$. \hfill $(2.3)$
	\end{itemize} 
	
	Now consider the set $J = \{x \in \mathcal{I} : \; m_x > 1 \}$. To prove our result, it is sufficient to show $|J| + m \leq  \displaystyle\mathop{\sum}_{a \in \mathcal{M}} \left\lfloor \frac{m_{a}}{2} \right\rfloor$ and $l + n \leq |J_2|$. By Theorem \ref{ch2-independence-monogenic}, $|\mathcal{I} \cap \langle b \rangle| \leq \left\lfloor \frac{m_b}{2}\right\rfloor$ for all $b \in B$ and by the proof of Theorem \ref{ch2-independence-monogenic}, $|\mathcal{I} \cap (\langle a \rangle \setminus \mathcal{K}_a)| \leq \left\lfloor \frac{m_a}{2}\right\rfloor$ for all $a \in A$. For $x \in S$ such that $m_x > 1$, we get $x \in \left(\displaystyle\mathop{\cup}_{a \in A}\langle a \rangle\right) \cup \left(\displaystyle\mathop{\cup}_{b \in B}\langle b \rangle\right)$. It follows that $|J| + m \leq \displaystyle\mathop{\sum}_{a \in A \cup B}\left\lfloor \frac{m_a}{2}\right\rfloor = \displaystyle\mathop{\sum}_{a \in \mathcal M}\left\lfloor \frac{m_a}{2}\right\rfloor$. 
	
	Next we show that $l + m \leq |J_2|$. We establish a one-one map from the set  $\mathcal{O} = \{x_1, x_2, \ldots, x_l, z_1, z_2, \ldots, z_m\}$ to some subset of $J_2$. In view of this, for each $p \in \mathcal{O}$, first we provide an element of $J_2$ corresponding to $p$. Let $p \in \mathcal{O}$ such that $p = z_k$ for some $k$, where $1 \leq k \leq n$. Since $z_k \in \langle c \rangle$ for some $c \in C \subseteq \mathcal{M}$, we have $m_c = 1$ and so $\langle c \rangle$ is a maximal cyclic subgroup of $S$. Choose $v \in X$ such that $\langle c \rangle = \langle v \rangle$. Then clearly, $v \in \overline{\mathcal{M}}$. If $v \in \langle x \rangle$ for some $x \in I_2$, then $\langle v \rangle = \langle x \rangle$ because $\langle v \rangle$ is a maximal monogenic subsemigroup of $S$. But $\langle v \rangle = \langle x \rangle$ is not possible because $m_v = 1$ and $m_x = 2$. It follows that for $p = z_k$ we have $v \in J_2$ such that $z_k \in \langle v \rangle$. We may now assume that $p' \in \mathcal{O}$ such that $p' = x_i$ for some $i$, where $1 \leq i \leq l$. Then $p' = x_i \in \langle a \rangle$ for some $a \in A$. Since $m_{x_i} = 1$, we get $x_i \in \mathcal{K}_a$. By the similar argument used in proof of Lemma \ref{ch1-exponent} (part (ii)), we get $\mathcal{K}_a \subseteq \langle d \rangle$ for some $d \in \overline{\mathcal{M}}$. If $d \in \langle x \rangle$ for some $x \in I_2$, then $x_i \in \langle x \rangle$. Since $m_x = 2$ , we have either $x \in \langle a \rangle$ for some $a \in A$ or $x \in \langle b \rangle$ for some $b \in B$. If $x \in \langle a \rangle$, then $x = a^j$ for some $j$. Clearly, $\langle d \rangle \subseteq \langle a^j \rangle$ and $\langle d \rangle = \mathcal{K}_a$ as $d \in \overline{\mathcal M}$. Then $\mathcal{K}_a \subseteq \langle a^j \rangle$ follows that $|K_{a^j}| = |\mathcal{K}_a| = \frac{r_a}{(j, \; r_a)} = r_a$ and so $(j, \; r_a) = 1$. Since $m_x = 2$ and $x = a^j$, by Lemma \ref{ch1-index-a^i}, we get $\left\lceil \frac{m_a}{2} \right\rceil \leq j \leq m_a - 1$. Therefore, $(j, r_a) > 1$ which is not possible. If $x \in \langle b \rangle$ for some $b \in B$, then $x_i \in \langle b \rangle$; a contradiction of $(2.1)$. Now choose $w \in X$ such that $\langle w \rangle = \langle d \rangle$. Thus, for $p' = x_i$ there exits $w \in J_2$ such that $x_i \in \langle w \rangle$. For each $p \in \mathcal{O}$, choose exactly one $s \in J_2$ such that $p \in \langle s \rangle$ for some $s \in J_2$. In view of this the assignment $f : p \mapsto s$ is a map from $\mathcal{O}$ to some subset $D$ of $J_2$. In fact, this map is one-one. For instance, if $p, q \in \mathcal{O}$ such that $pf = qf = s$ then $p, q \in \langle s \rangle$. It follows that $p \sim q$ (see Corollary \ref{ch2-Gamma(G)-complete}) which is a contradiction as $p, q \in I$. Consequently, from above $l + n = |\mathcal{O}| \leq |J_2|$.
\end{proof}

Now, we determine $\alpha(\Gamma(S))$, when $S$ is a completely $0$-simple semigroup. We require the following lemma.

\begin{lemma}\label{ch1-comp-0-simple-index}
	Let $x = (i, a, \lambda) \in \mathfrak{M}^0[G, I, \Lambda, P]$ such that $o(x)$ is finite.
	\begin{enumerate}[\rm (i)]
		\item If $p_{\lambda i} \ne 0$, then $m_x = 1$.

		\item If $p_{\lambda i} = 0$, then $m_x = 2$. Moreover, $x \in \mathcal{M}$.
	\end{enumerate}
\end{lemma}

\begin{proof}
	Suppose $x = (i, a, \lambda) \in \mathfrak{M}^0[G, I, \Lambda, P]$ such that $o(x)$ is finite and let $p_{\lambda i} \ne 0$. Then $(i, a, \lambda)^n = (i, (ap_{i\lambda})^{n-1}a, \lambda)$. Choose $n$ such that $n - 1$ is the order of $ap_{i\lambda}$, we get $(i, a, \lambda)^n = (i, a, \lambda)$. Consequently, $m_x = 1$. Now to prove (ii), we assume that $p_{\lambda i} = 0$. Then $x^2 = (i, a, \lambda)^2 = 0$ implies $m_x = 2$. Let if possible, $\langle x \rangle \subset \langle y \rangle$ for some $y = (j,b,\mu) \in \mathfrak{M}^0[G, I, \Lambda, P]$. Since $o(x)$ is finite so that $o(y)$ is finite. Note that $p_{\mu j} = 0$. Then $y^2 = 0$ gives $x = y$; a contradiction. 
\end{proof}

In view of Lemma \ref{ch1-comp-0-simple-index}, note that $I_2 = \{(i,a, \lambda) : \; a\in G \; {\rm and} \; p_{\lambda i} = 0 \}$. Note that $I_2 \subseteq \mathcal{M}$. If  $x = (i, a, \lambda) \in \mathfrak{M}^0[G, I, \Lambda, P] \setminus I_2$ then $m_x = 1$. Observe that $|I_2|  = \displaystyle\mathop{\sum}_{a \in I_2}\left\lfloor \frac{m_a}{2}\right\rfloor$. By Theorem \ref{ch2-inde-bound}, we get the independence number of $\Gamma(S)$, where $S$ is a finite completely $0$-simple semigroup, in the following corollary.

\begin{corollary}
	Let $S$ be a finite completely $0$-simple semigroup. Then  \[\alpha(\Gamma(S)) = |J_2| +  \displaystyle\mathop{\sum}_{a \in \mathcal{M}} \left\lfloor \frac{m_a}{2}\right\rfloor. \]
\end{corollary}

\begin{theorem}
	Let $S$ be a semigroup such that it satisfies one of the following condition
	\begin{enumerate}[\rm (i)]
		\item there exists $a \in S$ such that $o(a)$ is infinite
		
		\item $M = {\rm sup}\{m_a : a \in S \}$ is infinite
		
		\item the set $E(S)$ is infinite
		
		\item the set $\mathcal{M} = \{a \in S : \; \langle a \rangle \; {\rm is \; a \; maximal \; monogenic \; subsemigroup \; of} \; S \}$ is infinite.
	\end{enumerate}
	Then the independence number\index{independence number} of $\Gamma(S)$ is infinite.
\end{theorem}

\begin{proof}
	If there exists $a \in S$ such that $o(a)$ is infinite, then by the similar argument used in the proof  of Lemma \ref{ch2- adjacency-cond-more-than-one-index}, the set $\{a^p : p \; {\rm is \; prime} \}$ is an independent set of  $\Gamma(S)$. Now suppose that  $M = {\rm sup}\{m_a : a \in S \}$ is infinite. For $k \in \mathbb N$, we establish an independent set of size $\eta(k)$, where $\eta(k)$ is the number of primes less than $k$. Note that there exists $a \in S$ such that $m_a \geq k$. By Lemma \ref{ch2- adjacency-cond-more-than-one-index}, the set $\{a^p : \; p \; {\rm is \; a \; prime \; and} \; p < k \}$ is an independent set of size $\eta(k)$. Thus, $\alpha(\Gamma(S))$ is infinite. If $E(S)$ is infinite, then $\alpha(\Gamma(S))$ must be infinite because any two idempotent elements are not adjacent in $\Gamma(S)$. Further, we assume that the set $\mathcal{M}$ is infinite. 
	Then $S$ contains infinitely many maximal monogenic subsemigroup of $S$. Let $a, b \in \mathcal{M}$ such that $\langle a \rangle  \ne \langle b \rangle$. Clearly, $a \nsim b$. It follows that $\alpha(\Gamma(S))$ is infinite.
\end{proof}

\section{Acknowledgement}
The second author wishes to acknowledge the support of MATRICS Grant  (MTR/2018/000779) funded by SERB, India.

\end{document}